\documentclass[reqno,a4paper,12pt]{amsart}
\usepackage{amssymb,amsmath,amscd,amstext,amsthm,amsfonts,mathtools}
\usepackage[]{graphicx}
\usepackage{setspace} 
\usepackage{subfig}
\usepackage[mathscr]{eucal}   
\usepackage{stmaryrd}

\def\emph#1{{\em #1}}

\newcommand{\pa}{\partial}

\newcommand{\de }{\Delta}

\numberwithin{equation}{section}
\newtheorem{theorem}{Theorem}[section]

\newtheorem{corollary}[theorem]{Corollary}
\newtheorem{lemma}[theorem]{Lemma}

{\theoremstyle{definition}
{\newtheorem{remark}[theorem]{Remark}
\newtheorem{example}[theorem]{Example}

\newtheorem{defn}[theorem]{Definition}
\begin{document}
\title[Sprays and Dirac Structures]{Sprays and Dirac Structures}
\date{\today}
\keywords{second order differential equations; Dirac structures}

\author[Azizpour]{Esmaeil Azizpour}
\address{Department of Pure Mathematics, {Faculty of Mathematical  Sciences}\\
 University of Guilan, P.o. Box 1914,\\
Namjoo Street, Rasht, Iran}
\email{eazizpour@guilan.ac.ir}

\author[Moazzami]{Ghazaleh Moazzami}
\address{Department of Pure Mathematics, {Faculty of Mathematical  Sciences}\\
 University of Guilan, P.o. Box 1914,\\
Namjoo Street, Rasht, Iran}
\email{ghazale.moazzami@gmail.com}
\begin{abstract}
 We look at the the possibility  of a spray having a Hamiltonian description considering  Dirac structure as underlying geometric structure.  In a simpler scenario, we consider almost Dirac structure as an auxiliary object to find constants of motion for sprays. Staring with the horizontal distribution associated to a spray, we use gauge transformations generated by two forms in order to obtain constants of motion and possible Hamiltonian description. We also apply our approach to semi-spray with some modifications. 
 \end{abstract}

\maketitle

\section{Introduction}
  The original underlying geometry for  Hamiltonian systems is the symplectic geometry. {\it A symplectic structure} on the smooth manifold $M$ is given by a closed, non-degenerate two form $\omega$. The word non-degenerate means that the bundle maps 
\begin{equation}
\label{omega-sharp}
\omega^\sharp:TM\to T^\ast M,\quad X\to i_X\omega,
\end{equation}
 where $i_X\omega(Y)=\omega(X,Y)$ for every $X,Y\in TM$, is an isomorphism. A Hamiltonian vector field is defined using this isomorphism, i.e., given a Hamiltonian ${H:M\to\mathbb{R}}$ one assigns to it the Hamiltonian vector field defined by $X_H=(\omega^\sharp)^{-1}(d H)$. Even though, there are plenty of natural dynamical systems which are Hamiltonian in the symplectic setting, the non-degeneracy condition is not satisfied in many other situations. 
 One example, related to our work here, is the Euler-Lagrange equations arising from a singular Lagrangian.

 The first option to generalize the notion of a symplectic structure is to simply drop the non-degeneracy condition on the closed two form $\omega$, i.e., consider what is called {\it presymplectic structures}. In spite  of the fact that $\omega^\sharp$ turns degenerate, one still can define Hamiltonian vector field via equality $\omega^\sharp(X_H)=d H$. In this case the Hamiltonian vector field associated to a Hamiltonian will be unique up to addition of the element of the kernel of $\omega$. 

The second option to generalize the notion of a symplectic structure is to consider the inverse map $\pi_\omega^\sharp:=(\omega^\sharp)^{-1}$ and relax its  non-degeneracy property. Denoting the space of smooth function on $M$ by $C^\infty(M)$, {\it a poisson structure} is defined either by a skew-symmetric bracket $\{.,.\}:C^\infty(M)\times C^\infty(M)\to C^\infty(M)$ satisfying Leibniz's rule and Jacobi identity or by a bi-vector $\pi:T^\ast M\times T^\ast M\to\mathbb{R}$ with the property $[\pi,\pi]=0$, where $[.,.]$ is the well-known Schouten bracket. Similarly, one defines $\pi^\sharp:T^\ast M\to T^\ast M$. The inverse of the bundle map $\pi_\omega^\sharp$ defines a poisson structure since the closeness of $\omega$ yields the equality $[\pi_\omega,\pi_\omega]=0$. The poisson structure $\pi_\omega$ is non-degenerate but a general poisson structure might not be.  The poisson setting is more suitable for Hamiltonian description of dynamical systems in the sense that given any Hamiltonian $H:M\to\mathbb{R}$ its Hamiltonian vector field can be obtained directly by equation $X_H=\pi^\sharp(d H)$. The equation $[\pi,\pi]=0$ amounts to the fact that the characteristic distribution $C_\pi:={\rm Im}(\pi^\sharp)$ is integrable. Each leaf of the foliation integrating $C_\pi$ inherits a symplectic structure from the poisson structure $\pi$. Clearly, every Hamiltonian vector field $X_H=\pi^\sharp(d H)$ is tangent to this foliation. Restriction of a Hamiltonian vector field to a given leaf sets us back in the symplectic setting. 

 Dirac structures were introduced in \cite{MR998124,MR951168} as a way to unite and generalize both  presymplectic and poisson structures. This common framework relies on viewing presymplectic and poisson structures as subbundles of the vector bundle 
$$\mathbb{T}M:=TM\oplus T^\ast M,$$
defined by the graphs of $\omega^\sharp$ and $\pi^\sharp$.
A Dirac structure is a linear subbundle, $L$, of $\mathbb{T}M$ which is maximal isotropic with respect to the natural pairing 
\begin{equation}
\label{pairing}
\ll (X,\alpha),(Y,\beta)\gg=\frac{1}{2}\left(\beta(X)+\alpha(Y)\right),
\end{equation} 
 on $\mathbb{T}M$ and satisfies an integrability condition, see Section~\ref{sec:dirac} for more detail. The    graphs of $\omega^\sharp$ and $\pi^\sharp$ are maximal isotropic and closeness of the symplectic form and vanishing of the  Schouten bracket of the poisson structure with itself yield  integrability condition.

 A vector field $X$ is Hamiltonian with respect to Dirac structure $L$ if and only if there exist a function $H$ such that $(X,d H)\in L$.

A {\it semi-spray}, also referred to as second order differential equation (SODE), on $\mathbb{R}^n$ is a vector field $S$ defined on $T\mathbb{R}^n$ by 
\[S=\sum_{\alpha=1}^n y_\alpha\frac{\partial}{\partial x_\alpha}- 2\sum_{a=1}^n G^a(x,y)\frac{\partial}{\partial y_a},\]
where $(x_1,\ldots,x_n)$ is coordinates on $\mathbb{R}^n$ with its induced chart $(x_1,\ldots,x_n,$ $ y_1,\ldots,y_n)$ 
on $T\mathbb{R}^n$ and  $G^a, a=1,..,n$ are smooth functions. It is a well known fact that to every semi-spray one can associate a nonlinear connection. Via this association one, consequently, assigns a horizontal distribution to a given semi-spray, see Subsection~\ref{sode-connections} for more details. 

Our principal aim here is to find possible constants of motion and Hamiltonian descriptions for a given semi-spray. There are two motivations to consider Dirac structures, rather than symplectic or poisson structures, treating this problem. First one is the simple fact that Dirac  structures are more general and there is more possibility for a given vector field to become Hamiltonian with respect to a Dirac structure, see Remark~\ref{example-only-Dirac}. The second motivation is that symplectic and poisson manifolds have no local invariants including curvature. But, on the contrary, the concept of curvature can be defined for a Dirac manifold. In fact any Dirac manifold becomes a Lie algebroid considering the projection map to the tangent bundle as anchor map and the concept of connection and curvature is defined for Lie algebroids, see~\cite{MR1929305}. Here, we deal with semi-sprays and connections associated to them, so it makes more sense to consider objects that the concept of curvature is defined for them. 

The subject of our work closely parallels a situation in Lagrangian mechanics which is referred to as Helmholtz problem. Given equations of motion for a system, one seeks Lagrangian functions in which their Lagrangian equations are equivalent to the given equations of motion. A Lagrangian function is called a standard one if it is sum of kinetic energy and a potential function and is referred to as non-standard Lagrangian otherwise. Helmholtz problem has been studied by many researcher considering various methods. In regard to our work, in \cite{CR}, the authors study the second-order Riccati equation introducing the concept of non-standard Lagrangians for them. These systems have a preserved energy function i.e a constant of motion. It is shown there that the  values of the preserved energy function can be used as an appropriate parameter for characterizing the behavior of the solutions of the system. This can be considered as a motivation for seeking constants of motion for semi-prays even without considering the Hamiltonian description. 

Another relevant subject is Noether's theorem which connects 
symmetries of a system to constants of motions. It should be noted that in  various articles, it was shown that the constants of motion of a system does not necessarily result from the  Noether's theorem, see \cite{ho}, \cite{cr}. Our approach to constants of motion does not takes the symmetries into account. 

Given a semi-spray $S$, we take the horizontal distribution associated to it as starting point. This distribution yields an almost Dirac structure $L=D\oplus D^\circ$, see Subsection~\ref{sec:dirac} for more details. Given a two from $\omega$, via what is known as {\it gauge transformation}, one may obtain an other almost Dirac structure i.e.
\[L_{\omega}=\{(X,\alpha)\,|\, X\in D,\,\, (\alpha-i_X\omega)\in D^\circ\}.\]

If there exist a closed one-form $\alpha$ such that $(S,\alpha)\in L_\omega$ then $H$ such $dH=\alpha$ is a constant of motion for $S$. If $D$ is integrable then $L$ is a Dirac structure and further condition of $\omega$ being closed yields that $L_\omega$ is a Dirac structure as well. In this case $S$ has a Hamiltonian distribution with respect to $L_\omega$ having $H$ as its Hamiltonian. In our opinion, it is easier to obtain a constant motion of the system in this way than in conventional methods. We apply our method in some examples. 

{\bf Organization of the paper:} In Section~\ref{preliminaries}, we provide a brief introduction to the needed concepts. In Section~\ref{semi-sprays-and-dirac}, we present our results. In Section~\ref{examples}, we provide some examples.

\section{Preliminaries}
\label{preliminaries}

In this section, we present a brief introduction to Dirac structures, semi-sprays and non-linear connections. We will also present relations between these three objects.  

\subsection{Dirac structure} \label{sec:dirac} The vector bundle $\mathbb{T}M$ is called {\it big tangent bundle} or, in some literature, {\it Pontryagian bundle}. Denoting the natural pairing between the vector field $X\in\mathfrak{X}(M)$ and the one-form $\beta\in \Omega^1(M)$ by $\beta(X)$, a natural non-degenerate, symmetric and fiber-wise linear form is defined on $\mathbb{T}(M)$ by
\begin{equation}
\label{pairing}
\ll (X,\alpha),(Y,\beta)\gg=\beta(X)+\alpha(Y) \quad 
X,Y\in\mathfrak{X}(M)\quad  \alpha, \beta\in \Omega^1(M).
\end{equation} 
Let $L$ be a linear subbundle of $\mathbb{T}(M)$, its annihilator  with respect to the pairing $\ll.,.\gg$ is defined as 
\[L^\perp:=\{(X,\alpha)\in \mathbb{T}(M)\,\,| \,\,\ll(X,\alpha),(Y,\beta)\gg=0\quad \forall (Y,\beta)\in L\},\]
where by belonging to $L$, we mean being a section of $L$. We will use the same notion for the sections of $\mathbb{T}M$. 

The pairing\, $\ll.,.\gg$ \, is neither positive definite nor negative definite. As a consequence, for a given linear subbundle $L$ of $\mathbb{T}M$, the intersection $L\cap L^\perp$ can be non-empty. 
Having this in mind, a linear subbundle $L\subset V\oplus V^\ast$ is called \emph{isotropic} if $L\subseteq L^\perp$.
 
A linear subbundle $L\subset\mathbb{T}M$ is called an {\it almost Dirac structure} on the manifold $M$, or sometimes referred to as a {\it Lagrangian subbundle of $\mathbb{T}M$}, if and only if $L=L^\perp$.
\begin{defn}\label{almost-dirac}
A linear subbundle $L\subset\mathbb{T}M$ is called an {\it almost Dirac structure} on the manifold $M$, or sometimes referred to as a {\it Lagrangian subbundle of $\mathbb{T}M$}, if and only if $L=L^\perp$.
\end{defn}

The condition $L=L^\perp$ is also called {\it maximally isotropic}  since it yields the fact that ${\rm dim}(L)=n$ which is the maximum dimension of an isotropic subbundle. On the other hand, if ${\rm dim}(L)=n$ and $L\subseteq L^\perp$, then $L=L^\perp$, see~\cite{MR3098084}.

\begin{example}
\label{example-distribution}
Let $D$ be a  distribution on the manifold $M$ and $D^\circ$ its annihilator, then the subbundle $L_D:=D\oplus D^\circ$ clearly defines an almost Dirac structure. 
\end{example}
There is another structure on $\mathbb{T}(M)$ which is used to formalize an integrability condition i.e. the {\it Courant bracket}: 
\begin{equation}
\label{courant-bracket}
\llbracket(X,\alpha),(Y,\beta)\rrbracket=([X,Y],\mathcal{L}_X\beta-\mathcal{L}_Y\alpha+\frac{1}{2}d\left(\alpha(Y)-\beta(X)\right),
\end{equation}
where $(X,\alpha),(Y,\beta)\in\mathbb{T}M$. 

\begin{defn}
\label{dirac-structure}
An almost Dirac structure $L$, see Definition \ref{almost-dirac}, is called {\it  Dirac structure } if and only if  it is involutive with respect to Courant bracket $\llbracket.,.\rrbracket$ i.e.
\begin{equation}
\label{evolutive}
\llbracket(X,\alpha),(Y,\beta)\rrbracket\in L\quad\forall (X,\alpha),(Y,\beta)\in L.
\end{equation}
\end{defn}

Many examples of Dirac structures can be found in the literature, for example see \cite{MR998124,MR3098084,MR951168}.
The fact that Dirac structure unifies presymplectic and Poisson structure is verified by first and second items of the following example.
\begin{example}
\label{example-dirac}
let $M$ be an smooth manifold then 
\begin{itemize}
\item[i)] For any given presymplectic form $\omega$ on $M$, the graph of $\omega^\sharp$, see \eqref{omega-sharp}, i.e. $$L_\omega=\{(X,\omega^\sharp(X)) | X\in \mathfrak{X}(M)\}$$ 
 defines a Dirac structure on $M$. 
 \item[ii)] For any given poisson structure $\pi$ on the  smooth manifold $M$, the subbundle $L_\pi=\{(\pi^\sharp(\alpha),\alpha) | \alpha\in \Omega^1(M)\}$ defines a Dirac structure on $M$.
  \end{itemize}
  \end{example}

  In Example \eqref{example-distribution}, if the distribution $D$ integrates to a foliation then $L=D\oplus D^\circ$ is a Dirac structure, see~\cite{MR3098084}.
As its mentioned in~\cite[Example 3.6]{MR3098084}, given a closed two form $\omega$ the operation 
\[(X,\alpha)\to (X,\alpha +i_X\omega)\]
send Dirac structures to Dirac structures. This operation is called {\it gauge transformation}. Applying this operation to $L=D\oplus D^\circ$ we get:
 \begin{example}\label{the-one-we-use}
Given an integrable distribution $D$ and a closed form $\omega$, the linear subbundle 
\[\{(X,\alpha)\,|\, X\in D,\,\, (\alpha-i_X\omega)\in D^\circ\},\]
defines a Dirac structure. Dropping the integrability condition on $D$ or the closeness condition on $\omega$, we get an almost Dirac structure. 
\end{example}

{\bf Presymplectic leaves  and null distribution:} 
The bracket \eqref{courant-bracket} does not satisfy Jacobi identity for the sections of $\mathbb{T}M$, instead
\begin{equation}
\label{courant-jacobi}
\llbracket\llbracket a_1,a_2\rrbracket,a_3\rrbracket+c.p.=\frac{1}{3}d(\ll\llbracket a_1,a_2\rrbracket,a_3\gg+c.p.),
\end{equation}
for every $a_1,a_2,a_3\in \mathbb{T}M$, where $c.p.$ stands for cyclic permutation, so in general $\llbracket.,.\rrbracket$ does not give a Lie algebra structure to $\mathbb{T}M$. Restricting the Courant bracket to a Dirac  structure $L$, one see that the fact $L\subset L^\perp$ together with \eqref{courant-jacobi} yields
\begin{equation}
\label{evolutive2}
\llbracket\llbracket a_1,a_2\rrbracket,a_3\rrbracket+c.p=0\quad\forall a_1,a_2,a_3\in L,
\end{equation}
in other words, the Courant bracket restricted to the section of a Dirac  structure $L$ gives it a Lie algebra structure. This Lie algebra structure together with projection $Pr_1:L\to TM$ considered as {\it anchor map} gives $L$  what is called a {\it Lie algebroid structure}. As a consequence distribution $C:=Pr_{TM}(L)$ integrates to a possibly singular foliation of $M$. This foliation is called characteristic foliation of $L$.

Any leaf $\mathcal{O}$ of the characteristic foliation $C=Pr_{TM}(L)$ naturally inherits a closed two form $\omega_{L,\mathcal{O}}\in\Omega^2(\mathcal{O})$, defined at each point $x\in\mathcal{O}$ by 
\begin{equation}
\label{presymplectic-leave}
\omega_{L,\mathcal{O}}(X,Y)=\alpha(Y), ~~\mbox{where}\,X,Y\in T_x\mathcal{O}=C_x  ~~    \mbox{and}\, (X,\alpha)\in L_x.
\end{equation}
Definition~\eqref{presymplectic-leave} is independent of the choice of $\alpha$ since $L\subset L^\perp$. Closeness of $\omega_{L,\mathcal{O}}$ follows from the integrability of $L$.
Distribution 
\[K:=L\cap (TM\oplus\{0\}),\]
agrees, at each point, with the kernel of the leaf-wise two from $\Omega_{L,\mathcal{O}}$ and it is referred to as the kernel of the Dirac structure $L$. 

{\bf Hamiltonian vector fields:}
A function $H\in C^\infty(M)$ is called {\it admissible} on the Dirac structure, $L$, if there is a vector field $X_H$ such that
\[(X_H,d H)\in   L,\]
in which case $X_H$ is called Hamiltonian relative to $H$. Just as for the presymplectic case the vector field $X_H$ is unique up to addition of the elements of the kernel distribution $K$. When $K$ is regular, a function $H$ is admissible if and only if its differential annihilates $K$. 

\begin{remark}\label{example-only-Dirac}
Dirac structures are more general setting to study Hamiltonian vector fields than poisson and symplectic structures. 
Consider the singular poisson structure on $\mathbb{R}^3$ given by 
\[ \{x,y\}=\frac{1}{z}\,\,\{x,z\}=0,\quad \{y,z\}=0.\]
This poisson structure gives us Hamiltonian vector field $X_x=-\frac{1}{z}\frac{\partial}{\partial y}$ which is singular at $z=0$. We may rewrite this poisson structure as a Dirac structure which is smooth at $z=0$, i.e the one generated by 
 \[(\frac{\partial}{\partial y},-z dx),\quad (\frac{\partial}{\partial x)},z dy),\quad (0,dz).\]
 This Dirac structure has $z=$constant as presymplectic leaves and $\Omega=z dx\wedge dy$ as presymplectic form on them.
 The singular poisson structure above  appears in the study of guiding center motion in the plane.
 \end{remark}

%
\subsection{Non-linear connections and semi-sprays}\label{sode-connections}
In this subsection, we present needed concepts on semi-sprays, for more details about these concepts, one may consult \cite{MR1281613,MR2381561}. 

We consider $\mathbb{R}^n$ with coordinates $(x_1,..,x_n)$ and induced ones $(x_1,..,x_n,$ $y_1,...,y_n)$ on $T\mathbb{R}^n=\mathbb{R}^n\times\mathbb{R}^n$.
In most of the equations, the subscripts $i,j,..$ and $a,b,..$ will be used for the coordinates of the base and, respectively, for the coordinates on the fibers of $T\mathbb{R}^n$. The kernel of the differential of the projection $\pi:T\mathbb{R}^n\to \mathbb{R}^n$ determines a regular, $n$ dimensional, integrable distribution on the manifold $TQ$ which is called the vertical distribution. We will denote it by $VT\mathbb{R}^n$. The vertical vector field $\mathbb{C}=\sum_{a=1}^n y_a(\frac{\partial}{\partial y_a})$ is called {\it Liouville vector field}.

 The bundle map $J:TT\mathbb{R}^n\to TT\mathbb{R}^n$  defined by 
\begin{equation*}
J(\frac{\partial}{\partial x_\alpha})=\frac{\partial}{\partial y_\alpha},\quad J(\frac{\partial}{\partial y_a})=0,\quad\forall \alpha,a=1,...,n
\end{equation*}
is called the {\it tangent structure}. Clearly: ${\rm Ker} J={\rm Im}J=VTQ$, ${\rm rank}J=n$ and $J^2=0$. 

\begin{defn}\label{defn:sode1} A vector filed $S\in\mathcal{X}(TQ)$ is called a {\it semi-spray} iff $JS=\mathbb{C}$. A semi-spray is represented as follows:
 \begin{equation}
 \label{spray}
 S=\sum_{\alpha=1}^n y_\alpha\frac{\partial}{\partial x_\alpha}- 2\sum_{a=1}^n G^a(x,y)\frac{\partial}{\partial y_a},
 \end{equation}
 where $G^a,\,\,a=1,..,n,$ are smooth functions. Integral curves of the semi spray $S$ are solutions of following second order differential equations (SODE): 
 \begin{equation}
 \label{SODE}
 \frac{d^2 x_\alpha}{d t^2}+2G^\alpha(x,y)\frac{d x_\alpha}{dt}=0\quad \forall \alpha=1,..,n
\end{equation} 

A semi-spray $S$ is called a (full) {\it spray} iff the coefficient functions $G^a(x,y)$ are $2$-homogeneous in the second coordinate, i.e.
\begin{equation}\label{spray-condition}
G^a(x,\lambda y)=\lambda^2G^a(x,y)\quad\forall\lambda>0.
\end{equation}
\end{defn}

Vector fields $\frac{\partial}{\partial y_a},\,\,a=1,\ldots,n,$ constitute a basis for the vertical distribution $VT\mathbb{R}^n$. A supplementary distribution to the vertical distribution is defined as follows.
\begin{defn}\label{horizontal-distribution-to-semi-spray}
The distribution $H_S\subset TT\mathbb{R}^n$ generated by vector fields
\begin{equation}\label{horizontal-base}
\frac{\delta}{\delta x_i}=\frac{\partial}{\partial x_i}-\sum_{a=1}^n \frac{\partial G^a}{\partial y_\alpha}\,\frac{\partial}{\partial y_a},\quad\,i=1,\ldots,n,
\end{equation}
is called the {\it Horizontal} distribution associated to the semi-spray
\[ S=\sum_{\alpha=1}^n y_\alpha\frac{\partial}{\partial x_\alpha}- 2\sum_{a=1}^n G^a(x,y)\frac{\partial}{\partial y_a}.\]
Clearly, 
\begin{equation}\label{decomposition}T_uT\mathbb{R}^n=H_uT\mathbb{R}^n\oplus V_uT\mathbb{R}^n,\quad \mbox{for every}\,\, u\in T\mathbb{R}^n.\end{equation}
The basis $\{(\frac{\delta}{\delta x_i})_u,(\frac{\partial}{\partial y_a})_u\}_{i,a=1,..n}$ is referred to as the basis adapted  to the decomposition~\eqref{decomposition} or Berwald basis. The corresponding dual basis of $\{(\frac{\delta}{\delta x_i})_u,(\frac{\partial}{\partial y_a})_u\}_{i,a=1,..n}$ is $\{((d x_i)_u,\delta y_a:=d y_a+\sum_{i=1}^n \frac{\partial G^a}{\partial y_\alpha} d x_i\}_{i,a=1,..n}$. 
\end{defn}

Definition~\ref{horizontal-distribution-to-semi-spray} is motivated by a well-known fact that there is a bilateral relation between semi-sprays and non-linear connections. A non-linear connection, $N$, is defined by a $n$ dimensional distribution which is supplementary to the vertical distribution. Any given non-linear connection $N$ have a basis of the form
 \[
 \label{base-horizontal-expression}
 (\frac{\delta}{\delta x_i})_u=(\frac{\partial}{\partial x_i})_u-\sum_{a=1}^n N^a_i(u)(\frac{\partial}{\partial y_a})_u,\quad\,u\in T\mathbb{R}^n, 
 \]  
and clearly such a basis defines a non-linear connection. 
The above mentioned bilateral relation is as follows (for more details see \cite[Section 2.4]{MR2381561}):
\begin{itemize}
\item For a given semi-spray $S$ defined by functions $G^a(x,y)$, the coefficients 
\begin{equation}
\label{associated-connection}
N^a_\alpha(x,y)=\frac{\partial G^a}{\partial y_\alpha}
\end{equation} define a nonlinear connection on $TQ$. 
\item If $N$ is a nonlinear connection on $TQ$  then the semi spray $S$ defined by coefficients $$G^a(x,y)=\frac{1}{2}\sum_{\alpha=1}^n y_\alpha N^a_\alpha(x,y),$$ is the unique semi-spray which satisfies $S=h[\mathbb{C},S]$, where $h$ denotes the projection on the (Horizontal) distribution defining $N$. 
\end{itemize}
Furthermore, let $S$ be a semi-spray and $N$ the nonlinear connection associated to it, then $S$ is a spray if and only if the spray associated to $N$ coincides with $S$. This happens when
\[G^a(x,y)=\frac{1}{2}\sum_{\alpha=1}^n y_\alpha \frac{\partial G^a}{\partial y_\alpha},\]
which equivalent to \eqref{spray-condition}.
In other words:
\begin{lemma}\label{spray-is-horizontal} 
A spray $S$ belongs to the horizontal distribution of the nonlinear connection associated to it. 
\end{lemma}

A nonlinear connection $N$ is called {\it integrable} if and only if the corresponding horizontal distribution is involutive.

\begin{defn}
Let 
\[R^a_{ij}:=\frac{\delta}{\delta x_j}(N^a_i)-\frac{\delta}{\delta x_i}(N^a_j)\quad i,j,a=1,...,n\]
then the tensor $R:=\frac{1}{2}\sum_{ij,a=1}^nR^a_{ij}dx_j\wedge dx_i\otimes\frac{\partial}{\partial y_a}$ is called {\it the curvature tensor} of the nonlinear connection $N$. 
\end{defn}
A simple calculation shows that (see \cite{MR2381561} page 32)
\[[\frac{\delta}{\delta x_i},\frac{\delta}{\delta x_j}]=\sum_{a=1}^n R^a_{ij}\frac{\partial}{\partial y_a}.\]
So, we have
\begin{lemma}
\label{horizontal-integrability}
 The horizontal distribution defining a nonlinear connection $N$ is integrable if and only if its curvature tensor vanishes. 
 \end{lemma}

\section{Semi-sprays and Dirac structure}\label{semi-sprays-and-dirac}

In this section, we present our main result. Let $S$ be a semi-spray given by \eqref{spray}
 and consider the basis $\{(\frac{\delta}{\delta x_i})_u,(\frac{\partial}{\partial y_a})_u\}_{i,a=1,..,n}$ where 
 $$(\frac{\delta}{\delta x_i})_u=(\frac{\partial}{\partial x_i})_u-\sum_{a=1}^n \frac{\partial G^a}{\partial y_\alpha}(\frac{\partial}{\partial y_a})_u,$$
 as defined in Section~\ref{sode-connections}. 

We start with a spray $S$. By Lemma~\ref{spray-is-horizontal}  spray $S$ is horizontal i.e. it belongs to the characteristic distribution of the almost Dirac structure 
\begin{equation}
 L:=HT\mathbb{R}^n\oplus (HT\mathbb{R}^n)^\circ,
 \end{equation}
see Example \ref{example-distribution}. We are going to apply the gauge transformations mentioned in Example~\ref{the-one-we-use} to this almost Dirac structure in order to get Constants of motion. 

\begin{theorem}\label{first-result}
Given a spray $S$ on $\mathbb{R}^n$, if there exist a two form $\omega$ and a closed one-form $\alpha$ such that 
\begin{equation}\label{main-requirement}
(\alpha-i_S\omega)\in (HT\mathbb{R}^n)^\circ,
\end{equation}
then the function $H\in C^\infty(\mathbb{R}^n)$ such that $d H=\alpha$ is a constant of motion for $S$. Furthermore, if $S$ is $R$-flat, i.e. the horizontal distribution $(HT\mathbb{R}^n)^\circ$ is integrable, and $\omega$ is closed then $S$ is Hamiltonian with respect to Dirac structure defined by 
\[\{(X,\alpha)\,|\, X\in (HT\mathbb{R}^n),\,\,(\alpha-i_S\omega)\in (HT\mathbb{R}^n)^\circ)\}.\]
\end{theorem}
\begin{proof}
The proof is simply a conclusion of Examples~\ref{example-distribution} and \ref{the-one-we-use}. 
\end{proof}

Rewriting Theorem~\ref{first-result} in local coordinates $(x,y)$,
 \begin{corollary}\label{first-cor}
  Given spray $S=\sum_{\alpha=1}^n y_\alpha\frac{\partial}{\partial x_\alpha}- 2\sum_{a=1}^n G^a(x,y)\frac{\partial}{\partial y_a}$, if there exist an anti-symmetric matrix valued function $A(u)$ and functions $f_1,...,f_n$ such that the one form 
  \[\alpha=A.\begin{pmatrix}
  y_1\\ \vdots\\y_n\\-2 G^1(x,y)\\\vdots\\ -2G^n(x,y)
  \end{pmatrix}+\begin{pmatrix}
  \sum_{i=1}^n f_i\frac{\partial G^i}{\partial y_1}\\\vdots\\ \sum_{i=1}^n f_i\frac{\partial G^i}{\partial y_n}\\f_1\\\vdots\\f_n
  \end{pmatrix},\]
  is closed  then The function $H$ such that $d H=\alpha$ is a constant of motion for $S$. Furthermore, if $S$ is $R$-flat and the $2$-form represented by $A$ is closed (for example $A$ is constant) then $S$ has a Hamiltonian description.
 \end{corollary}
 
When $S$ is not a spray we will need to choose  a distribution as well. 
\begin{theorem}\label{second-result}
Given a semi-spray $S$ on $\mathbb{R}^n$, let $D$ be a distribution that contains $S$, then if there exist a two form $\omega$ and a closed one-form $\alpha$ such that 
\begin{equation}
(\alpha-i_S\omega)\in D^\circ,
\end{equation}
then the function $H\in C^\infty(\mathbb{R}^n)$ such that $d H=\alpha$ is a constant of motion for $S$. Furthermore, if the distribution $D$ is integrable, and $\omega$ is closed then $S$ is Hamiltonian with respect to Dirac structure defined by 
\[\{(X,\alpha)\,|\, X\in D,\,\,(\alpha-i_S\omega)\in D^\circ)\}.\]
\end{theorem}

\begin{proof}
The proof is the same as Theorem~\ref{first-result}.
\end{proof}

In choosing the distribution $D$ one may use a slight perturbation of the horizontal distribution. 

\section{Existence of non-standard Hamiltonians}\label{examples}

In this section we provide examples, a spray and one semi-spray. 

\begin{example} 
Let $Q=\mathbb{R}^2-\{0\}$ with coordinates ${x_1},{x_2}$. 
Given 
$$ S=y_1 \frac{\partial }{\partial {x_1}}+y_2 \frac{\partial }{\partial {x_2}}-2y_2^2\frac{\partial }{\partial y_2}.$$
It is clear that $S$ is a spray. Since $G^1=0, G^2=y_2^2$, thus 
$\frac{\delta}{\delta {x_1}}=\frac{\partial}{\partial x_1},   
 \frac{\delta}{\delta {x_2}}=\frac{\partial}{\partial x_2}-2y_2 \frac{\partial}{\partial y_2}$, also $\delta y_1=dy_1, \delta y_2=dy_2+2y_2dx_2.$

 In this case $HT\mathbb{R}^2=<S,  \frac{\delta }{\delta {x_1}}>$ and $(HT\mathbb{R}^2)^\circ=<\delta y_1, \delta y_2>$. Note that, for $i=1,2, \delta y_i(S)=0.$   

 Let $\omega=\Upsilon_{12}dx_1\wedge dx_2+\sum_{s=1}^2\sum_{t=1}^2\Psi_{st} dx_s\wedge \delta y_t+\Omega_{12} \delta y_1\wedge d\delta y_2.$ According to the theorem, $H$ is a constant of motion for $S$ if $dH -i_s\omega  \in (HT\mathbb{R}^n)^\circ.$ In local coordinates, this means that 
for $\mu $ and $\nu$ 
\begin{subequations}\label{ex1-1}
\begin{align}
\sum_{i=2}^3(\frac{\delta H}{\delta {x_i}}dx_i &+ \frac{\pa H}{\pa {y_i}}\delta y_i)   -  \{  y_1 \Upsilon_{12}dx_2 - y_2 \Upsilon_{12}dx_1 + \sum_{s=1}^2\sum_{t=1}^2 \Psi_{st}(y_s \delta y_t)\} 
\nonumber \\
=&
 \mu dy_1+ \nu (dy_2+2y_2dx_2). 
\end{align}
\end{subequations}
Comparing coefficients of the separate basis 1-forms on the two sides of this equation, we see that 

\begin{subequations} \label{ex1}
\begin{align}
\frac{\partial H }{\partial {x_1}}+y_2 \Upsilon_{12}=0,  
\\
\frac{\delta H}{\delta {x_2}}-y_1\Upsilon_{12}= 0,   
 \\
\frac{\partial H}{\partial {y_1}}-(\Psi_{11}y_1+\Psi_{21}y_2)=\mu,   
 \\
\frac{\partial H}{\partial {y_2}}-(\Psi_{12}y_1+\Psi_{22}y_2)=\nu.    
\end{align}
\end{subequations}
Let 
\[v_1:=y_1, v_2:=-\frac{(2y_2 x_1-y_1)}{y_2y_1}
,v_3:= x_2-{1\over 2}ln(\frac{y_1}{y_2}) \] and $H_1(v_1,v_2, v_3)$ be an arbitrary differentiable function of these variables. By using Maple program, we can get a solution of \eqref{ex1} as follows:
 \begin{subequations} \label{solexa1}
\begin{align}
& H(x,y,z,v) = H_1(v_1,v_2, v_3),\\
& \Upsilon_{12} = \frac{2}{y_1y_2}\frac{\partial H_1}{\partial v_2} \\
& \mu =-\frac{1}{2y_1^2}\left( -2\frac{\partial H_1}{\partial v_1}y_1^2-4\frac{\partial H_1}{\partial v_2}x_1+\frac{\partial H_1}{\partial v_3}y_1+2y_1^3\psi_{11}+2y_2 \psi_{21} y_1^2\right) \\
&\nu =\frac{1}{2y_2^2}\left( -2\frac{\partial H_1}{\partial v_2}+\frac{\partial H_1}{\partial v_3}y_2-2y_2^2y_1\psi_{12}-2y_2^3 \psi_{22}\right),
\end{align}
\end{subequations}
where $\psi_{ij}= \psi_{ij}(x_1,x_2,y_1,y_2), i,j=1,2$ are arbitrary functions. 
 %

Here we consider a special case of
\eqref{solexa1} in which we take $\Upsilon_{12}=0, \Omega_{12}=0$ and $\Psi_{ij}=0, i,j=1,2$ except for $\Psi_{11}$. In this case, the presymplectic form takes the form

\[  \omega=\Psi_{11} dx_1\wedge \delta y_1\]
and  the almost Dirac structure associated to this system is 
\[ span \{  (S, dH), (\frac{\delta}{\delta {x_1}}, \Psi_{11}\delta y_1),  (\frac{\pa}{\pa {y_1}}, -\Psi_{11}dx_1), (\frac{\pa}{\pa {y_2}}, 0)   \}.\]
\end{example}

%
%
\begin{example} In this example we consider the first  class of equations of motion  of the form 
\begin{equation}\label{special eq-1}
\ddot{x}_1 = - \frac{f' (x_1)y_1^2}{f(x_1)}, \quad
\ddot{x}_2 = -2f(x_1).
 \end{equation}
where $f(x_1)$ is an arbitrary  differentiable function  of $x_1$. 
In this case we assume that $G^1=\frac{y_1^2  f'(x_1)}{2f(x_1)}, G^2=f(x_1), \quad$  
 $ D=<S,  \frac{\delta }{\delta {x_1}}>$ and $D^\circ=<\delta y_1, \delta y_2+2\frac{f}{y_2}dx_2>$. Note that 
\[ \{\Delta x_1:=dx_1-\frac{y_1}{y_2}dx_2, \de S:=\frac{1}{y_2}dx_2, \de y_1:= \delta y_1,  \de y_2:=\delta y_2+2\frac{f}{y_2}dx_2\} \]
 is the annihilator of $\{   S,  \frac{\delta }{\delta {x_1}}, \frac{\pa }{\pa {y_1}}, \frac{\pa }{\pa {y_2}} \}$.
%
 Let 
\begin{eqnarray*} \omega =&\Upsilon_{12}\de S \wedge \de x_1  +\Psi_{11} \de S \wedge \de y_1+\Psi_{12} \de S \wedge \de y_2  \nonumber \\
&+\Psi_{21} \de x_1 \wedge \de y_1+\Psi_{22} \de x_1 \wedge \de y_2+\Omega_{12}\de y_1  \wedge \de y_2  \nonumber
\end{eqnarray*}
 be a two form on $\mathbb{R}^2.$
 According to the theorem, $H$ is a constant of motion for $S$ if $dH -i_s\omega  \in D^\circ.$ In local coordinates, this means that 
for $\mu $ and $\nu$ 
\begin{eqnarray*}
S( H) \de S &+&\frac{\delta H}{\delta {x_1}}\de x_1 + \frac{\pa H}{\pa {y_1}}\de y_1 + \frac{\pa H}{\pa {y_2}}\de y_2 \\
  & -&  \{  \Upsilon_{12} \de x_1  +\Psi_{11}  \de y_1+\Psi_{12}  \de y_2 \}
%
=  \mu \de y_1+\nu \de y_2. \nonumber
\end{eqnarray*} 
The coefficients of $\de x_1$ and $\de S$  in the above equation are zero,  so
\begin{subequations}\label{ex2-2}
\begin{align}
  &S( H) =0, \label{ex2-21}\\
 &\frac{\delta H}{\delta {x_1}} -\Upsilon_{12}=0 \label{ex2-22}\\
 &\frac{\pa H}{\pa {y_1}}-\Psi_{11} =\mu,\label{ex2-23} \\
  & \frac{\pa H}{\pa {y_2}}-\Psi_{12} =\nu. \label{ex2-24} 
\end{align}
\end{subequations}
Let $v=4G^2 x_2+y_2^2$ and $H_1(v)$ be an arbitrary differentiable function of this variable. Then a solution  of  \eqref{ex2-2} (using Maple program) is:
\[ 
H(x_1,x_2,y_1,y_2) = H_1(v), \mu = -\psi_{11}, \nu= 2H_1'y_2-\psi_{12}, \Upsilon_{12}=0.\]

In addition to the above answer, there are other answers to Equation \eqref{ex2-2}.
For example, let  $\mu=0, \nu=0.$ The choice of  $ \Psi_{11}=2f$ and $\Psi_{12}=0, \Upsilon_{12}=0$ yield 
 $ H_{y_1}=2f,  H_{y_2}=0,$ thus $H=2fy_1+g(x_1,x_2)$.   If the rest of the coefficients of $\omega$  are equal to zero, except for 
  $\Omega_{12}=1$, then from \eqref{ex2-22} we have $\frac{\pa g}{\pa {x_1}}=0$. Finally from  \eqref{ex2-21}  we can conclude that $\frac{\pa g}{\pa {x_2}}=0$. These calculations show that  $g$ is a constant function and   $H=2f(x_1) y_1 +constant$  is a      constant of motion for $S$. Under our assumption, the local form of our 2-form is 
\[  \omega =2f \de S \wedge \de y_1 + \de y_1\wedge \de y_2\]
and  the almost Dirac structure associated to this system is 
\[ span \{  (S, dH), (\frac{\delta}{\delta {x_1}}, 0),  (\frac{\pa}{\pa {y_1}}, dy_2), (-\frac{\pa}{\pa {y_2}}, \de y_1)   \}.\]

\end{example}

The second special class of the equations of motions which admit non-standard Hamiltonian is the following constrained  problem.
\begin{example} 
Consider the  constrained mechanical system $(\mathbb{R}^3, L, C)$ described by the Lagrangian function $L=\frac{1}{2}({y_1}^2+{y_2}^2+{y_3}^2)-mgx_3$ and the quadratic constraint
 $$C: \Phi=a^2({y_1}^2+{y_2}^2)-{y_3}^2.$$
 Let $Q$ be the constraint submanifold of  $T\mathbb{R}^3$ defined by 
\[ Q = \{ ({x_1},{x_2},{x_3},{y_1},{y_2},{y_3})\in  T\mathbb{R}^3| \Phi ({x_1},{x_2},{x_3},{y_1},{y_2},{y_3})=0 \hspace{0.2cm} {\text and} \hspace{0.2cm} {y_3}\neq 0 \}.\]
The equations of motion of this  constrained problem (see \cite{MR1445410} )
 are 
\begin{equation}\label{special eq-2}
\ddot{x}_1 = -\frac{ga^2}{1+a^2}\frac{ y_1}{ y_3}, \quad \quad   \\
\ddot{x}_2= -\frac{ga^2}{1+a^2}\frac{ y_2}{ y_3},\quad \quad  \\
\ddot{x}_3= -\frac{ga^2}{1+a^2}.
\end{equation}
The semi-spray associated to these equations is 
\begin{equation}\label{cms-1}
  S=	 y_1\frac{\partial}{\partial x_1}+y_2\frac{\partial}{\partial x_2}+y_3\frac{\partial}{\partial x_3}- 2A\frac{y_1}{y_3} \frac{\partial}{\partial y_1}- 2A\frac{y_2}{y_3}\frac{\partial}{\partial y_2}- 2A\frac{\partial}{\partial y_3}, 
\end{equation}
where $A=\frac{ga^2}{2(1+a^2)}$. For this semi-spray, $G^1=A\frac{y_1}{y_3} , G^2=A\frac{y_2}{y_3} , $ and $G^3=A .$ Thus
$\frac{\delta}{\delta {x_1}}=\frac{\partial}{\partial x_1}- \frac{A}{y_3} \frac{\partial}{\partial y_1},  
\frac{\delta}{\delta {x_2}}=\frac{\partial}{\partial x_2}- \frac{A}{y_3} \frac{\partial}{\partial y_2}$ and 
$\frac{\delta}{\delta {x_3}}=\frac{\partial}{\partial x_3}+\frac{Ay_1}{y_3^2} \frac{\partial}{\partial y_1}+ 
\frac{Ay_2}{y_3^2} \frac{\partial}{\partial y_2}$.

 Let $D=\{ \frac{\delta }{\delta {x_1}}, \frac{\delta }{\delta {x_2}}, S\}.$ We choose $\{\delta y_1+\frac{2Ay_1}{y_3^2}dx_3, \quad \delta y_2-\frac{y_2}{y_3}\delta y_3\}$ as a basis for $D^\circ$. 

 Let $\omega=\sum_i\sum_j\Upsilon_{ij}dx_i\wedge dx_j+\sum_{s=1}^3\sum_{t=1}^3\Psi_{st} dx_s\wedge \delta y_t+\sum_i\sum_j\Omega_{ij} \delta y_i\wedge d\delta y_j, \quad for ~ i<j ~ and ~   i,j=1,...,3.$ According to the theorem, $H$ is a constant of motion for $S$ if $dH -i_s\omega  \in D^\circ.$ In local coordinates, this means that 
for $\mu $ and $\nu$ 
\begin{eqnarray}\label{le-1}
\sum_{i=1}^3 &&(\frac{\delta H}{\delta {x_i}}dx_i        + \frac{\pa H}{\pa {y_i}}\delta y_i)   -  \{  y_1 \Upsilon_{12}dx_2 - y_2 \Upsilon_{12}dx_1  + y_1 \Upsilon_{13}dx_3 - y_3 \Upsilon_{13}dx_1 \nonumber \\
&& + y_2 \Upsilon_{23}dx_3 - y_3 \Upsilon_{23}dx_2 + \sum_{s=1}^3\sum_{t=1}^3 \Psi_{st}(y_s \delta y_t +2G^tdx_s) 
-2G^1\Omega_{12}\delta y_2\nonumber \\
&&+2G^2\Omega_{12}\delta y_1  -2G^1\Omega_{13}\delta y_3+2G^3\Omega_{13}\delta y_1   -2G^2\Omega_{23}\delta y_3+2G^3\Omega_{23}\delta y_2    \} \nonumber \\
=&&
 \mu ( \delta y_1+\frac{2Ay_1}{y_3^2}dx_3)+ \nu (\delta y_2-\frac{y_2}{y_3}\delta y_3). 
\end{eqnarray} 
Now, equating coefficients of the separate basis 1-forms  of both sides of \eqref{le-1} yields 6 equations:
\begin{eqnarray*}
&&\frac{\partial H }{\partial {x_1}}-  \frac{A }{y_3} \frac{\partial H }{\partial y_1}+y_2 \Upsilon_{12} +y_3 \Upsilon_{13}-2 \sum_{t=1}^3G^t\Psi_{1t}=0, \nonumber \\
&&\frac{\partial H }{\partial {x_2}}-  \frac{A }{y_3} \frac{\partial H }{\partial y_2} -y_1 \Upsilon_{12} +y_3 \Upsilon_{23} -2\sum_{t=1}^3G^t\Psi_{2t}=0, \nonumber \\
&&\frac{\partial H }{\partial {x_3}}+  \frac{Ay_1 }{y_3^2} \frac{\partial H }{\partial y_1}+ \frac{Ay_2 }{y_3^2} \frac{\partial H }{\partial y_2} -y_1 \Upsilon_{13} -y_2 \Upsilon_{23} -2\sum_{t=1}^3G^t\Psi_{3t}=\frac{2Ay_1 \mu  }{y_3^2} \nonumber \\
&&\frac{\partial H }{\partial {y_1}}-\sum_{t=1}^3y_t\Psi_{t1} -2G^2\Omega_{12}-2G^3\Omega_{13}=\mu \nonumber \\
&&\frac{\partial H }{\partial {y_2}}-\sum_{t=1}^3y_t\Psi_{t2} +2G^1\Omega_{12}-2G^3\Omega_{23}=\nu \nonumber\\
&&\frac{\partial H }{\partial {y_3}}-\sum_{t=1}^3y_t\Psi_{t3}+2G^1\Omega_{13}+2G^2\Omega_{23}=\frac{-\nu y_2}{y_3}.\nonumber
\end{eqnarray*}
 By  using Maple program, we can get a solution of \eqref{le-1} as follows:
$H({x_1},{x_2},{x_3},{y_1},{y_2},{y_3})$ can be an arbitrary differentiable function and in this case 
\begin{eqnarray*}
&&\nu =(2\frac{\partial H }{\partial {y_2}}G^2y_3-y_1y_3\frac{\partial H }{\partial {x_1}}+2G^3\frac{\partial H }{\partial {y_3}}y_3+2\mu Ay_1+2\frac{\partial H }{\partial {y_1}}G^1y_3 \nonumber \\
     &&\quad -2\mu G^1y_3  -\frac{\partial H }{\partial {x_3}}y_3^2-y_2 \frac{\partial H }{\partial {x_2}}y_3 ) 
                   /  {(-2y_2G^3+2y_3G^2)},  \nonumber \\
&&\psi_{13}= {1/2}  (\frac{\partial H }{\partial {x_1}}y_3-2G^1\psi_{11}y_3+y_3^2\Upsilon_{13}
        +y_2\Upsilon_{12}y_3  \nonumber\\
         &&\quad -2G^2 \psi_{12}y_3-A\frac{\partial H }{\partial {y_1}})  /  {(G^3y_3)},\nonumber  \\
&&\psi_{23}={1/2}\frac{(\frac{\partial H }{\partial {x_2}}y_3-2G^1\psi_{21}y_3-A\frac{\partial H }{\partial {y_2}}-                         
        y_1\Upsilon_{12}y_3+y_3^2\Upsilon_{23}-2G^2\psi_{22}y_3)}{G^3y_3}, \nonumber \\
&&\psi_{31}=\frac{\frac{\partial H }{\partial {y_1}}-y_1\psi_{11}-y_2\psi_{21}-2G^2\Omega_{12}-G^3\Omega_{13}-\mu }{y_3}, \nonumber  \\
&&\psi_{32}={1/2} (y_1\frac{\partial H }{\partial {x_1}}y_3+y_2\frac{\partial H }{\partial {x_2}}y_3+\frac{\partial H }{\partial {x_3}}y_3^2     -2\frac{\partial H }{\partial {y_1}}G^1y_3-2\frac{\partial H }{\partial {y_2}}G^3y_2 
     \nonumber  \\
     && -2\frac{\partial H }{\partial {y_3}}G^3y_3 +  4G^1 ( y_3G^2-y_2G^3)\Omega_{12}+(-4G^3G^2y_3+4(G^3)^2y_2)\Omega_{23}\nonumber   \\
      && -2y_1(y_3G^2-y_2G^3)\psi_{12}  +(-2y_2G^2y_3+2y_2^2G^3)\psi_{22}  \nonumber  \\ 
       &&  -2\mu (-y_3G^1+Ay_1)   )   /   {(y_3(y_3G^2-y_2G^3 ))}, \nonumber 
\end{eqnarray*}
\begin{eqnarray*}
\psi_{33}&=&{1/2}( ((2G^1y_2G^3+Ay_1G^2)y_3-Ay_1y_2G^3)\frac{\partial H }{\partial {y_1}} \nonumber \\ 
  && \quad +y_2(G^2(A+2G^3)y_3-AG^3y_2)\frac{\partial H }{\partial {y_2}}  -y_1\frac{\partial H }{\partial {x_1}}y_3^2G^2-y_2\frac{\partial H }{\partial {x_2}}G^2y_3^2  \nonumber  \\ 
  &&\quad -y_2G^3\frac{\partial H }{\partial {x_3}}y_3^2    +2G^3\frac{\partial H }{\partial {y_3}}G^2y_3^2 -y_1y_3^2(y_3G^2-y_2G^3)\Upsilon_{13}  \nonumber \\ 
 &&\quad+(G^3y_2^2y_3^2-y_2y_3^3G^2)\Upsilon_{23}  +4G^1G^3y_3(y_3G^2-y_2G^3)\Omega_{13}\nonumber \\
  &&\quad +(4G^3(G^2)^2y_3^2    -4(G^3)^2G^2y_3y_2)\Omega_{23}+2G^1y_1y_3(y_3G^2-y_2G^3)\psi_{11} \nonumber  \\      
  &&\quad +2G^2y_1y_3(y_3G^2-y_2G^3)\psi_{12} 
 +2(G^1y_3(y_3G^2-y_2G^3)\psi_{21}  \nonumber \\
   && \quad+G^2y_3(y_3G^2-y_2G^3)\psi_{22} \nonumber \\
   && \quad +G^3\mu (-y_3G^1+Ay_1))y_2)  /  {(G^3y_3^2(y_3G^2-y_2G^3))}.\nonumber 
\end{eqnarray*}
where $\psi_{ij}= \psi_{ij}(x_1,x_2,y_1,y_2), i,j=1,2$ are arbitrary functions.

 We 
can  introduce a constant of motion $H$ for $S$ by
appropriate choice of $\omega,$ different from previous answers: 
 if we put $\Upsilon_{ij}=0$ for each $i,j$ and $\Psi_{11}=\Psi_{12}=\Psi_{13}=\Psi_{21}=\Psi_{22}=\Psi_{23}=0$,  then we can  consider  a  special  choice of $\omega$.
Putting
\[\frac{\partial H }{\partial {x_1}}=\frac{\partial H }{\partial {x_2}}=0, \frac{\partial H }{\partial {y_1}}=\frac{\partial H }{\partial {y_2}}=0.\] 
With the above assumptions, we have 
 $\omega=\Psi_{3t} dx_3\wedge \delta y_t+\Omega_{ij} \delta y_i\wedge d\delta y_j, \quad for ~ i<j ~ and ~  s, t, i,j=1,...,3.$ 
Now from \eqref{le-1} we have 
\begin{eqnarray*}
 \frac{2Ay_1 }{y_3^2}\mu=\frac{\delta H}{\delta {x_3}}-2\sum_{t=1}^3\Psi_{3t}G^t= \frac{\partial H }{\partial x_3}-2\sum_{t=1}^3\Psi_{3t}G^t, \nonumber \\
-\Psi_{31}y_3-2G^2\Omega_{12}-2G^3\Omega_{13}=\mu \nonumber \\
-\Psi_{32}y_3+2G^1\Omega_{12}-2G^3\Omega_{23}=\nu \nonumber \\
\frac{\partial H}{\partial {y_3}}-\Psi_{33}y_3+2G^1\Omega_{13}+2G^2\Omega_{23}=-\frac{y_2 }{y_3}\nu.\nonumber
\end{eqnarray*}
Now,  if we put $\Psi{33}=2$ 
  also $\Psi_{32}=\Psi_{31}=0, \Omega_{12}=\Omega_{13}=0$, which leads to 
\begin{equation}\label{cms-2} \omega= 2dx_3\wedge \delta y_3+\Omega_{23} \delta y_2 \wedge \delta y_3, \end{equation}
 for arbitrary constant function  $\Omega_{23}$, then, assuming $\mu=0, \nu=-2G^3 \Omega_{23}$,
we find that under these conditions $\frac{\partial H}{\partial {x_3}}=4A$ and $\frac{\partial H}{\partial {y_3}}=2y_3.$
Thus   $H=y_3^2+4A{x_3}$ is a constant of the motion of $S$. (Compare this function with the similar one obtained in \cite{MR1445410},  Example 4.3). 

To achieve  almost Dirac structures 
 ,  we refer
to the  equation \eqref{cms-2} in which we set
the coefficient  $\Omega_{23}=0$.   So  $\omega= 2dx_3\wedge \delta y_3$ is a closed 2-form. The Dirac structure associated to $D$ and $\omega$  spanned  by the sections
\[ \{ (S, dH), (\frac{\delta }{\delta {x_1}}, 0), (\frac{\delta }{\delta {x_2}}, 0), (0, \delta y_1+\frac{2Ay_1}{y_3^2}dx_3), (0,  \delta y_2-\frac{y_2}{y_3}\delta y_3\})\}.\]

  We rewrite the semi-spray
as
\[
 S=y_1\frac{\delta }{\delta {x_1}}-\frac{Ay_1}{y_3}\frac{\partial }{\partial {y_1}}+y_2 \frac{\delta }{\delta {x_2}}-\frac{Ay_2}{y_3}\frac{\partial }{\partial {y_2}}+y_3\frac{\partial }{\partial {x_3}} -2A \frac{\partial }{\partial {y_3}}.
 \]
Then 
\begin{subequations}
\begin{align}
 [S, \frac{\delta }{\delta {x_1}}] &=- \{ \frac{\delta y_1}{\delta {x_1}}\frac{\delta }{\delta {x_1}}+\frac{\delta (-\frac{Ay_1}{y_3})}{\delta {x_1}}\frac{\partial }{\partial {y_1}}
+\frac{\delta y_2 }{\delta {x_1}}\frac{\delta }{\delta {x_2}}+\frac{\delta (-\frac{Ay_2}{y_3})}{\delta {x_1}}  \frac{\partial }{\partial {y_2}} \nonumber\\
 & \quad +\frac{\delta  y_3 }{\delta {x_1}}\frac{\partial }{\partial {x_3}} +\frac{\delta (-2A) }{\delta {x_1}} \frac{\partial }{\partial {y_3}}  \}.\nonumber
\end{align}
\end{subequations}
This shows that $D$ is not integrable. 
\end{example}
\begin{example}
Let $S$ be a spray  with zero coefficients on $\mathbb{R}^n$, i.e. for all $a, G^a=0.$ 

 Let $\omega=\sum_{a=1}^{n}dx_a\wedge d y_a.$ For $i=1, ..., n-1$ and $a=1, ..., n$,  
\[\{\de x_i=dx_i-\frac{y_i}{y_n}dx_n,  \hspace{0.2cm} \de S=\frac{dx_n}{y_n},\hspace{0.2cm} dy_a\}\]
is the annihilator of 
\[\{S, \hspace{0.2cm} \frac{\pa}{\pa x_i}, \hspace{0.2cm}\frac{\pa}{\pa y_a}\}.\]
 Using the necessary  condition for a constant of motion function, we have 
\[0= dH-i_S\omega=S(H)\de S+ \sum_{i=1}^{n-1} \frac{\pa H}{\pa x_i}\de x_i+ \sum_{a=1}^{n}\frac{\pa H}{\pa y_a}dy_a -  \sum_{a=1}^{n}y_a d y_a.\]
This means that
\begin{subequations}
\begin{align}
S(H) &=0\\
 \frac{\pa H}{\pa x_i}&=0 \\
 \frac{\pa H}{\pa y_a}-y_a &=0.
\end{align}
\end{subequations}
Thus  $H=\frac{1}{2}\sum_{a=1}^n {y}_a^2 $  is  the    constant of motion of $S$ and 
the almost Dirac structure associated to $S$  is  spanned  by the sections
\begin{equation}
L_{\rm diag}:={\rm Span}\{(\frac{\delta}{\delta x_\alpha},\delta y_\alpha),(\frac{\partial}{\partial y_\beta},-d x_\beta)|\alpha=1,..,n\}. \label{diag-almost-dirac}\end{equation}
\end{example} 

\section*{Acknowledgments}

The authors would like to thank Hassan Najafi Alishah for useful discussions.



\end{document}